\documentclass[]{interact}

\usepackage{epstopdf}
\usepackage[caption=false]{subfig}
\usepackage[numbers,sort&compress]{natbib}
\renewcommand\bibfont{\fontsize{10}{12}\selectfont}
\makeatletter
\def\NAT@def@citea{\def\@citea{\NAT@separator}}
\makeatother

\theoremstyle{plain}
\newtheorem{theorem}{Theorem}[section]

\newtheorem{corollary}[theorem]{Corollary}
\newtheorem{proposition}[theorem]{Proposition}

\theoremstyle{definition}
\newtheorem{definition}[theorem]{Definition}
\newtheorem{example}[theorem]{Example}

\theoremstyle{remark}

\begin{document}
	
	
	\title{On relationships between vector variational inequalities and optimization problems using convexificators on Hadamard manifold}
	
	\author{
		\name{Nagendra Singh\textsuperscript{a}\thanks{ Email: gg2763@myamu.ac.in; singh.nagendra096@gmail.com}, Akhlad Iqbal \textsuperscript{a*}\thanks{ *Corresponding author: Email:  akhlad6star@gmail.com, akhlad.mm@amu.ac.in} and Shahid Ali \textsuperscript{a}\thanks{ Email: shahid\_rrs@yahoo.com}}
		\affil{\textsuperscript{a}Department of Mathematics, Aligarh Muslim University Aligarh, India}
	}

	\maketitle
	
	\begin{abstract}
		\noindent
		An important concept of convexificators has been extended to Hadamard manifolds in this paper. The mean value theorem for convexificators on the Hadamard manifold has also been derived. Monotonicity of the bounded convexificators has been discussed and an important characterization for the bounded convexificators to be $\partial_{*}^{*}$-geodesic convexity has been derived. Furthermore, a vector variational inequalities problem using convexificators on Hadamard manifold has been considered. In addition, the necessary and sufficient conditions for vector optimization problem in terms of Stampacchia and Minty type partial vector variational inequality problem ($\partial_{*}^{*}$-VVIP) has been derived.
	\end{abstract}
	{\bf{AMS SUBJECT CLASSIFICATIONS}}
	\noindent
	47J20, 49J52, 53C22, 58E17
	\section{Introduction}
	
	In 1980, Giannessi \cite{giannessi1980theorems} described the variational inequality problems (VIP) in vector form using Stampacchia \cite{stampacchia1964formes} and Minty \cite{minty1967generalization} Euclidean space setup. He also showed relations between solutions of Minty vector variational inequalities and efficient solution of a differential convex vector optimization problems. Since, then relations between nonsmooth vector variational inequalities and non-smooth vector optimization problem has been shown intensively, see \cite{ansari2017nonsmooth,mishra2016minty}. In 1994, Demyanov \cite{demyanov1994convexification} proposed the concept of convexificators in order to the generalization of upper convex and lower concave approximations. Later, Demyanov and Jeykumar \cite{demyanov1997hunting} evaluated convexificators for positively homogeneous and locally Lipschitz functions. Jeykumar and Luc \cite{jeyakumar1999nonsmooth} defined non-compact convexificators and presented several calculus rules for calculating convexificators. For more details related to the convexificators see \cite{demyanov2000exhausters} and the references therein. Laha and Mishra \cite{laha2017vector} studied the convexity for vector valued functions in terms of convexificators and the monotonicity of the corresponding convexificators. Moreover, Laha et al. \cite{laha2017vector} formulated the vector VIP of Stampacchia and Minty type using convexificators on Euclidean spaces.\\
	
	\noindent
	In other way, numerous methods and techniques of non-linear analysis have been expanded from linear space to non-linear spaces, in particularly on Hadamard manifolds. In order to go for further in the details of the convex theory, the variational inequality and related topics, see \cite{sakai1992riemannian,udriste2013convex,singh2023bi,ferreira2005singularities,nemeth2003variational,chen2016vector}. Nemeth \cite{nemeth2003variational} extended the VIP on Hadamard manifolds and also studied their existence. Later, Chen and Huang \cite{chen2016vector} shown the relations between vector variational inequality problems (VVIP's) and vector optimization problems (VOP's) on Hadamard manifolds. Further, Chen \cite{chen2020existence} and Jayswal et. al. \cite{jayswal2021vector} investigated the extension of solutions to the problem of this kind under appropriate conditions.\\
	
	\noindent
	Most of the well-known sub-differentials including the Clark, Michel Penot and Mardukhovich sub-differentials may strictly include the convex hull of the convexificators of the locally Lipschitz functions, which provide sharper results for VOP in terms of convexificators, see \cite{giannessi1980theorems,giannessi1998minty,demyanov1994convexification,ferreira2005singularities,laha2017vector,mishra2016minty}.Motivated by the above work, we define the concept of convexificator on Hadamard manifold and discuss several rules for monotonicity of $ \partial_{*}^{*}f$ and $ \partial_{*}^{*}$-convexity. We prove the mean value theorem using convexificators on Hadamard manifold. Afterward, we extend the VVIP to Hadamard manifold using convexificators and further use it as a tool for finding the solution of VOP.
	\noindent
	\section{Preliminaries}
	In this section, firstly we recall some concepts, definitions, properties and results about Riemannian manifolds which help us in proving several important characterizations. For more study of Riemannian manifolds, see: \cite{boothby2003introduction,sakai1992riemannian,udriste2013convex}.\\
	\noindent 
	Let $\mathbb{R}^{m}$ be a $m$-dimensional Euclidean space and $ \mathbb{R}_{+}^{m}$ be its non-negative orthant.\\
	Let $ p=(p_{1},p_{2},...,p_{m}) $ and $ q=(q_{1},q_{2},...,q_{m}) $ be the two vectors in $ \mathbb{R}^{m} $. Then,
	\begin{eqnarray*}
		&&p\leqq q\Leftrightarrow p_{l}\leq q_{l}\quad \mbox{for}~l=1,2,...,m\quad\Leftrightarrow p-q\in -\mathbb{R}_{+}^{m};\\
		&&p\leq q \Leftrightarrow p_{l}\leq q_{l}\quad \mbox{for}\quad l=1,2,...,m \quad \mbox{and} \quad p\neq q \quad\Leftrightarrow p-q \in -\mathbb{R}_{+}^m;\\&&
		p<q \Leftrightarrow p_{l}< q_{l}\quad \mbox{for}\quad l=1,2,...,m \quad\Leftrightarrow p-q \in -int\mathbb{R}_{+}^m;
	\end{eqnarray*}
	\begin{definition}{\rm \cite{jeyakumar1999nonsmooth}}
		Let $\Psi:\mathbb{R}^{m}\longrightarrow \mathbb{R}\cup\{+\infty\}$ such that for $p\in \mathbb{R}^{m},$ $\Psi(p)$ be finite. The lower and upper Dini derivative of $\Psi$ at $p$ in the given direction of $w\in \mathbb{R}^{m}$ are defined respectively, as follows
		\begin{eqnarray*}
			\Psi^{-}(p,w):=\liminf_{t\downarrow 0}\frac{\Psi(p+tw)-\Psi(p)}{t}
		\end{eqnarray*}
		\begin{eqnarray*}
			\Psi^{+}(p,w):=\limsup_{t\downarrow 0}\frac{\Psi(p+tw)-\Psi(p)}{t}
		\end{eqnarray*}
	\end{definition}
 \begin{definition}{\rm \cite{jeyakumar1999nonsmooth}}
 	Let $\Psi:\mathbb{R}^{m}\longrightarrow \mathbb{R}\cup \{+\infty\}$ such that for $p\in \mathbb{R}^{m},$ $\Psi(p)$ be finite. Then, the function $\Psi$ is said to have:
 	\begin{enumerate}
 		\item an upper convexificator $\partial^{*}\Psi(p)\subset\mathbb{R}^{m}$ at $p\in \mathbb{R}^{m}$, iff $\partial^{*}\Psi(p)$ is closed and for each $w\in \mathbb{R}^{m}$, one has
 		\begin{eqnarray*}
 			\Psi^{-}(p;w)\leq \sup_{\xi\in\partial^{*}\Psi(p)}\left\langle\xi, w\right\rangle.
 		\end{eqnarray*}
 		\item  a lower convexificator $\partial_{*}\Psi(p)\subset\mathbb{R}^{m}$ at $p\in \mathbb{R}^{m}$, iff $\partial_{*}\Psi(p)$ is closed and for each $w\in \mathbb{R}^{m}$, one has
 		\begin{eqnarray*}
 			\Psi^{+}(p;w)\geq \inf_{\xi\in\partial_{*}\Psi(p)}\left\langle\xi, w\right\rangle.
 		\end{eqnarray*}
 		\item  a convexificator $\partial^{*}_{*}\Psi(p)\subset\mathbb{R}^{m}$ at $p\in \mathbb{R}^{m}$, iff $\partial^{*}_{*}\Psi(p)$ is both upper and lower convexificators of $\Phi$ at $p$.  
 	\end{enumerate}
 \end{definition}
 \noindent
 Let $ \mathcal{M} $ be $m$-dim Riemannian manifold with Levi-civita  (or Riemannian) connection $ \nabla $. The scalar product on $ T_{p}\mathcal{M} $ with the norm $\| \cdot\|$ is denoted by $\left\langle \cdot,\cdot\right\rangle $.\\
 For any $ p,q\in \mathcal{M} $, Let $ \gamma_{pq} :[0,1]\longrightarrow{\mathcal{M}} $ be a piece-wise smooth curve joining $ p $ to $ q $. Then the arc length of $ \gamma_{pq}(t) $ is
 \begin{equation*}
 	L(\gamma_{pq}):= \int_{0}^{1}\|\dot\gamma_{pq}(t)\|dt
 \end{equation*}
 where $ \dot\gamma_{pq}(t) $ is the tangent vector to the curve $ \gamma_{pq} $.\\
 \noindent
 If a smooth curve satisfies the conditions
 $\gamma_{pq}(0)=p~,\gamma_{pq}(1)=q~~\mbox{and}~ \nabla_{\dot\gamma_{pq}}{\dot\gamma_{pq}}=0~\mbox{on}~[0,1]$,
 called geodesic on manifold. If we take two points $p, w\in \mathcal{M}$, $ P_{w,p} $ denotes the parallel transport from $ T_{p}\mathcal{M} $ to $ T_{w}\mathcal{M} $. \\
 By the Hopf-Rinow theorem, we know that, if any two points on $ \mathcal{M} $ can be joined by a minimal geodesic then $ \mathcal{M} $ is complete Riemannian manifold and arc-length of the geodesic is called Riemannian distance between $ p ~\mbox{and} ~q $ and it is defined as $ d(p,q)= \inf_{\gamma_{pq}}L(\gamma_{pq}) $.\\
 Now, recall that a function $ \Psi:\mathcal{M}\longrightarrow\mathbb{R} $ is said to be Lipschitz on the given subset $ \mathcal{K} $ of $\mathcal{M}$ if $\exists$ a $ \lambda\geq 0 ,$ such that 
 \begin{eqnarray*}
 	|\Psi(p)-\Psi(q)|\leq\lambda d(p,q),~~\forall~p,q\in \mathcal{K}.
 \end{eqnarray*}
 A function $ \Psi:\mathcal{M}\longrightarrow\mathbb{R} $ is said to be locally Lipschitz function at point $ p_{o}\in \mathcal{M} $, if $\exists$ $ \lambda(p_{o})\geq 0 $ such that the above inequality satisfies with $ \lambda=\lambda(p_{o}) $ for any $ p,~q $ in a neighborhood of $ p_{o} $.
 Let us start some basic definitions of generalized derivative for locally Lipschitz function on $\mathcal{M}$.\\
 \begin{definition}{\rm \cite{motreanu1982quasi}}
 	Let $\Psi:\mathcal{M}\longrightarrow \mathbb{R}$ be a locally Lipschitz function. Let $ p, q\in \mathcal{M} $, the generalized directional derivative $\Psi^{o}(p;v)$ of $\Psi$ at point $ p $ in the direction $ v\in T_{p}\mathcal{M} $ defined as
 	\begin{eqnarray*}
 		\Psi^{\circ}(p;v)=\limsup_{q\rightarrow p, t\downarrow 0,q\in \mathcal{M}}\frac{\Psi \circ\Phi^{-1}(\Phi(q)+td\Phi(p)(v))-\Psi \circ\Phi^{-1}(\Phi(q))}{t}
 	\end{eqnarray*} 
 	where $\Phi:U\subseteq\mathcal{M}\longrightarrow\Phi(U)\subseteq\mathbb{R}^{m}$ be a homeomarphism, that is $ (U, \Phi) $ is the chart about the point $ p $.
 \end{definition}
 \begin{definition}{\rm \cite{motreanu1982quasi}}
 	Let $ \Psi:\mathcal{M}\longrightarrow \mathbb{R} $ be a locally Lipschitz function on Riemannian manifold. Then, the generalized gradient of $ \Psi $ at the point  $ q\in \mathcal{M} $ is the subset $ \partial_{c}\Psi(q) $ of $ T_{q}^{*}\mathcal{M}\cong T_{q}\mathcal{M}$ defined as 
 	\begin{eqnarray*}
 		\partial_{c}\Psi(q)=\{\xi\in T_{q}\mathcal{M}:\Psi^{\circ}(q;v)\geq\left\langle\xi,q\right\rangle,~~\forall v\in T_{q}\mathcal{M}\}.
 	\end{eqnarray*}
 	
 \end{definition}
 \begin{definition}{\rm\cite{jost1997nonpositive}(Hadamard manifold):} 
 	A complete, simply connected Riemannian manifold which has non-positive sectional curvature is called Hadamard manifold, and we denote it by $\mathbb{H}$ through out the paper.
 \end{definition}
 \noindent
 \begin{proposition}{\rm \cite{docarmo1992riemannian}}
 	Let $p$ be any point of Hadamard manifold $ \mathbb{H} $. Then, $ \exp_{p}:T_{p}\mathbb{H}\longrightarrow \mathbb{H} $ is a diffeomorphism. And for any $ p,q\in \mathbb{H} $, there exists a unique minimal geodesic $ \gamma_{pq}$ joining $ p $ to $ q $ such that 
 	\begin{equation*}
 		\gamma_{pq}(t)=\exp_{p}(t\exp_{p}^{-1}q),\quad\forall~ t\in [0,1].
 	\end{equation*}  
 \end{proposition}
 \begin{definition}{\rm \cite{udriste2013convex}}
 	A set $ \mathcal{K}\subseteq\mathbb{H} $ is said to be geodesic convex (GC) if for any two points $ p,q\in \mathcal{K}$, $\exp_{x}t\exp_{p}^{-1}q\in \mathcal{K}$.
 \end{definition}
 \begin{definition}{\rm \cite{udriste2013convex}}
 	Suppose $ \mathcal{K}\subseteq\mathbb{H} $ is a GC set. Then $ \Psi:\mathcal{K}\longrightarrow\mathbb{R} $ is said to be convex function if for every $ p,q\in \mathcal{K} $,
 	\begin{eqnarray*}
 		\Psi(\exp_{p}t\exp_{p}^{-1}q)\leq t\Psi(p)+(1-t)\Psi(q),~~\forall ~t\in[0,1].
 	\end{eqnarray*}
 \end{definition}
 \begin{definition}{\rm \cite{ansari2017nonsmooth}}
 	Let $\Psi:\mathbb{H}\longrightarrow\bar{\mathbb{R}}:= \mathbb{R}\cup\{+\infty\}$ be an extended real valued function on $\mathbb{H}$ and $ p $ be a point where $ \Psi$ is finite.
 	\begin{enumerate}
 		\item The Dini-lower directional derivative at point $ p\in \mathbb{H} $ in the direction $ v\in T_{p}\mathbb{H} $ is defined as 
 		\begin{eqnarray*}
 			\Psi^{-}(p;v):=\liminf_{t\rightarrow 0^{+}}\frac{\Psi(\exp_{p}tv)-\Psi(p)}{t}.
 		\end{eqnarray*}
 		\item  The Dini-upper directional derivative at point  $ p\in \mathbb{H} $ in the direction $ v\in T_{p}\mathbb{H} $ is defined as
 		\begin{eqnarray*}
 			\Psi^{+}(p;v):=\limsup_{t\rightarrow 0^{+}}\frac{\Psi(\exp_{p}tv)-\Psi(p)}{t}.
 		\end{eqnarray*}
 	\end{enumerate}
 \end{definition}
 \noindent
 As discussed in \cite{ansari2017nonsmooth}, for a fixed $ s\in (0,1) $, we take a point $ w= \gamma_{pq}(s)=\exp_{p}(s\exp_{p}^{-1}q) $ on the geodesic $ \gamma_{pq}:[0,1]\longrightarrow H $, which divides the geodesic into two parts, the first part can be written as
 \begin{equation*}
 	\gamma_{wp}(t)=\gamma_{pq}(-st+s)=\exp_{p}(-st+s)\exp_{p}^{-1}q,\quad\forall\quad t\in [0,1];
 \end{equation*}
 that is,
 \begin{equation}\label{e1}
 	\exp_{w}(t\exp_{w}^{-1}p)=\exp_{p}(-st+s)\exp_{p}^{-1}q,\quad\forall\quad t\in [0,1];
 \end{equation}
 and the second part can be written as 
 \begin{equation*}
 	\gamma_{wq}=\gamma_{pq}((1-s)t+s)=\exp_{p}(((1-s)t+s)\exp_{p}^{-1}q),\quad\forall\quad t\in [0,1];
 \end{equation*}
 that is,
 \begin{equation}\label{e2}
 	\exp_{w}(t\exp_{w}^{-1}q)=\exp_{p}(((1-s)t+s)\exp_{p}^{-1}q),\quad\forall\quad t\in [0,1].
 \end{equation}
 From (\ref{e1}) and (\ref{e2}), we get
 \begin{equation}\label{e3}
 	\exp_{w}^{-1}p=-sP_{w,p}\exp_{p}^{-1}q,
 \end{equation}
 \begin{equation}\label{e4}
 	\exp_{w}^{-1}q=(1-s)P_{w,p}\exp_{p}^{-1}q.
 \end{equation}
 Similarly, we have 
 \begin{equation}\label{e5}
 	\exp_{w}^{-1}p=sP_{w,q}\exp_{q}^{-1}p.
 \end{equation}
 
\section{Convexity and Monotonicity of convexificators}

In this section, we firstly prove the mean value theorem for convexificators on Hadamard manifold. We extend the notions of convexity and monotonicity of vector valued function using convexificators to Riemannian manifold particularly Hadamard manifold and establish some relations between them.
\begin{definition}
	Let $ \Psi:\mathbb{H}\longrightarrow\bar{\mathbb{R}} $ be a extended real-valued function, $ p\in\mathbb{H} $ and let $ \Psi(p) $ be finite.
	\begin{enumerate}
		\item The function $ \Psi $ is said to have an upper convexificator $ \partial^{*}\Psi(p)\subset T_{p}\mathbb{H} $ at a point $ p\in \mathbb{H} $, iff $ \partial^{*}\Psi(p) $ is closed and for each $ v\in T_{p}\mathbb{H} $ such that
		\begin{eqnarray*}
			\Psi^{-}(p;v)\leq \sup_{\xi\in\partial^{*}\Psi(x)}\left\langle\xi;v \right\rangle
		\end{eqnarray*}
		\item The function $ \Psi $ is said to have a lower convexificator $ \partial_{*}\Psi(p)\subset T_{p}\mathbb{H} $ at point $ p\in \mathbb{H} $, iff $ \partial_{*}\Psi(p) $ is closed and for each $ v\in T_{p}\mathbb{H} $ such that
	\begin{eqnarray*}
		\Psi^{+}(p;v)\geq \inf_{\xi\in\partial_{*}\Psi(p)}\left\langle\xi;v \right\rangle
	\end{eqnarray*}
	\item The function $ \Psi $ is said to have a convexificator 
	$\partial_{*}^{*}\Psi(p)\subset T_{p}\mathbb{H}$ at point $p\in \mathbb{H}$, iff $\partial_{*}^{*}\Psi(p)$ is both upper and lower convexificators of $\Psi$ at $p$.
\end{enumerate}
\end{definition}
\begin{theorem}{\bf [Mean Value Theorem]}\label{[Mean Value Theorem]}
	Suppose $ \mathcal{K}(\neq\phi)\subseteq\mathbb{H} $ is a GC set. Let $ p,q\in \mathcal{K} $ and let $\Psi:\mathcal{K}\longrightarrow \bar{\mathbb{R}}:=\mathbb{R}\cup\{-\infty, +\infty\}$ be finite and continuous. Suppose that, for each $t\in (0,1) $, $ z(t):=\exp_{p}(t\exp_{p}^{-1}q) $, $\partial^{*}\Psi(z)~\mbox{and}~\partial_{*}\Psi(z)$ are respectively upper and lower convexificators of $\Psi$. Then, there exists $ w(t)\in (p,q) $ and a sequence $\{\xi_{k}\}\subset \mbox{co}(\partial^{*}\Psi(w)\cup\partial_{*}\Psi(w))$ such that
	\begin{equation*}
		\Psi(q)-\Psi(p)=\lim_{k\rightarrow \infty}\left\langle\xi_{k};P_{w,p}\exp_{p}^{-1}q\right\rangle;
	\end{equation*}
	or
	\begin{equation*}
		\Psi(q)-\Psi(p)=\left\langle\xi;P_{w,p}\exp_{p}^{-1}q\right\rangle;
	\end{equation*}
\end{theorem}
\begin{proof}
	Consider a function $ \rho:[0,1]\longrightarrow\mathbb{R} $ such that 
	\begin{equation*}
		\rho(t):= \Psi(\exp_{p}t\exp_{p}^{-1}q)-\Psi(p)+t(\Psi(p)-\Psi(q)).
	\end{equation*}
	Here, $ \rho $ is continuous on $ [0,1] $ and $ \rho(0)=\rho(1)=0 $. Then, $\exists$ $\mu\in(0,1)$ such that $\mu$ is the extremum point of $\rho$. Define
	\begin{equation*}
		w(\mu)=\exp_{p}\mu\exp_{p}^{-1}q.
	\end{equation*}
	Without loss of generality, let $ \mu $ is the minimal point of $\rho$, then using the necessary condition of minimal point, for each $ v\in\mathbb{R} $,
	\begin{equation*}
		\rho_{d}^{-}(\mu;v)\geq0.
	\end{equation*}
	Since,
	\begin{equation*}
		\rho_{d}^{-}(\mu;v):=\liminf_{k\rightarrow0^{+}}\frac{\rho(\mu+kv)-\rho(\mu)}{k};
	\end{equation*}
	Therefore, we have
	\begin{equation*}
		\liminf_{k\rightarrow0^{+}}\frac{\Psi(\exp_{p}(\mu+kv)\exp_{p}^{-1}q)-\Psi(\exp_{p}\mu\exp_{p}^{-1}q)}{k}+v(\Psi(p)-\Psi(q))\geq0.
	\end{equation*}
	Since,
	\begin{eqnarray}\label{e6.}
		\exp_{p}(\mu + kv)\exp_{p}^{-1}q=\exp_{p}\left (-\mu\left(\frac{kv}{-\mu}\right)+\mu\right)\exp_{p}^{-1}q.
	\end{eqnarray}
	Now, suppose
	\begin{eqnarray*}
		\frac{kv}{-\mu}=\lambda\quad(\mbox{say}).
	\end{eqnarray*}
	Therefore, equation (\ref*{e6.}) becomes
\begin{eqnarray*}
	\exp_{p}(\mu + kv)\exp_{p}^{-1}q&=&\exp_{p}\left (-\mu\lambda+\mu\right)\exp_{p}^{-1}q=\gamma_{wp}(\lambda).\\
	&=&\exp_{w}\lambda\exp_{w}^{-1}p\\
	&=&\exp_{w}k\left(\frac{v}{-\mu}\right)\exp_{w}^{-1}p.
\end{eqnarray*}
Hence, from above inequality
\begin{eqnarray*}
	&&\liminf_{k\rightarrow0^{+}}\frac{\Psi \left(\exp_{w}k\left(\frac{v}{-\mu}\exp_{w}^{-1}p\right)\right)-\Psi(\exp_{p}\mu\exp_{p}^{-1}q)}{k}+v\left(\Psi(p)-\Psi(q)\right)\geq0.\\	
	&&\liminf_{k\rightarrow0^{+}}\frac{\Psi \left(\exp_{w}kv'\right)-\Psi(w)}{k}+v\left(\Psi(p)-\Psi(q)\right)\geq0.\\	
	&&\Psi_{d}^{-}(w;v')+v(\Psi(p)-\Psi(q))\geq0.\\	
	&&\Psi_{d}^{-}\left(w;\frac{v}{-\mu}\exp_{w}^{-1}p\right)+v(\Psi(p)-\Psi(q))\geq0.	
\end{eqnarray*}
We know that 
\begin{eqnarray*}
	-\frac{1}{\mu}\exp_{w}^{-1}p=P_{w,p}\exp_{p}^{-1}q.
\end{eqnarray*}
	This implies that
	\begin{equation*}
		\Psi_{d}^{-}(w;vP_{w,p}\exp_{p}^{-1}q)\geq v(\Psi(q)-\Psi(p));
	\end{equation*}
	Now, putting $ v=1~\mbox{and}~v=-1 $, we get
	\begin{equation*}
		-\Psi_{d}^{-}(w;P_{w,p}\exp_{p}^{-1}q)\leq\Psi(q)-\Psi(p)\leq\Psi_{d}^{-}(w;P_{w,p}\exp_{p}^{-1}q);
	\end{equation*}
	since $\partial^{*}\Psi(w)$ is an upper convexificator of $\Psi$ at $ w $, we have
	\begin{equation*}
		\inf_{\xi\in\partial^{*}\Psi(w)}\left\langle\xi; P_{w,p}\exp_{p}^{-1}q\right\rangle\leq\Psi(q)-\Psi(p)\leq\sup_{\xi\in\partial^{*}\Psi(w)}\left\langle\xi;P_{w,p}\exp_{p}^{-1}q\right\rangle.
	\end{equation*}
	Then, this inequality follows that $\exists$ a sequence \{$\xi_{k}\}\subset\mbox{co}(\partial^{*}\Psi)$ such that
	\begin{equation*}
		\Psi(q)-\Psi(p)=\lim_{k\rightarrow0}\left\langle\xi_{k};P_{w,p}\exp_{p}^{-1}q\right\rangle.
	\end{equation*}
	On the other hand, if $\mu$ is the maximal point of $\rho$, then using the same arguments as above, we get the conclusion. Hence,
	\begin{equation*}
		\Psi(q)-\Psi(p)=\left\langle\xi;P_{w,p}\exp_{p}^{-1}q\right\rangle.
	\end{equation*} 
\end{proof}
\begin{definition}
	Suppose $\mathcal{K}(\neq\phi)\subseteq\mathbb{H}$ is a GC set and $\Psi:\mathcal{K}\longrightarrow\mathbb{R}^{m}$ be a function such that $ \Psi_{i}:\mathcal{K}\longrightarrow\mathbb{R} $ is locally Lipschitz at $ \bar{p}\in \mathcal{K}\subseteq \mathbb{H} $ and admit a bounded convexificator $ \partial_{*}^{*}\Psi_{i}(\bar{p}) $ at a point $ \bar{p} $ for all $ \forall~i\in M=\{1,2,...m\}$. Then $ \Psi $ is said to be: 
	\begin{enumerate}
		\item $ \partial_{*}^{*}$-convex at point $ \bar{p} $ over $ \mathcal{K} $, iff for any $ p\in \mathcal{K} $ and $ \xi^{*}\in \partial_{*}^{*}\Psi(\bar{p}) $ such that
		\begin{equation*}
			\Psi(p)-\Psi(\bar{p})\geqq \left\langle\xi^{*}; \exp_{\bar{p}}^{-1}p \right\rangle_{m}.
		\end{equation*}
		\item strictly  $ \partial_{*}^{*}$-convex at point $ \bar{p} $ over $ \mathcal{K} $, iff for any $ p\in \mathcal{K} $ and $ \xi^{*}\in \partial_{*}^{*}\Psi(\bar{p}) $ 
		\begin{equation*}
			\Psi(p)-\Psi(\bar{p})~>~ \left\langle\xi^{*}; \exp_{\bar{p}}^{-1}p \right\rangle_{m}.
		\end{equation*}
	\end{enumerate}
	Where 
	\begin{equation*}
		\xi^{*}:= (\xi_{1}^{*},\xi_{2}^{*},...,\xi^{*}),
	\end{equation*}
	\begin{equation*}
		\partial_{*}^{*}\Psi(\bar{p}):= \partial_{*}^{*}\Psi_{1}(\bar{p})\times...\times \partial_{*}^{*}\Psi_{m}(\bar{p}),
	\end{equation*}
	\begin{equation*}
		\left\langle\xi^{*}; \exp_{\bar{p}}^{-1}p\right\rangle_{m}:= (\left\langle\xi_{1}^{*}; \exp_{\bar{p}}^{-1}p\right\rangle, \left\langle\xi_{2}^{*}; \exp_{\bar{p}}^{-1}p\right\rangle,...,\left\langle\xi_{m}^{*}; \exp_{\bar{p}}^{-1}p\right\rangle).
	\end{equation*}
\end{definition}
\begin{definition}
	Let $ \Psi:=(\Psi_{1}, \Psi_{2},..., \Psi_{m}):\mathcal{K}\longrightarrow\mathbb{R}^{m} $ be a vector valued function such that $ \Psi_{i}:\mathcal{K}\longrightarrow \mathbb{R} $ are locally Lipschitz on $ \mathcal{K}\subseteq \mathbb{H} $ and admit a bounded convexificator $ \partial_{*}^{*}\Psi_{i}(p) $ for all $ p\in \mathcal{K} $ and $\forall~i\in M=\{1,2,...m\} $. Then, $ \partial_{*}^{*}\Psi $ is said to be: 
	\begin{enumerate}
		\item monotone on $\mathcal{K}$, iff for any $ p,q\in \mathcal{K} $, $ \xi\in \partial_{*}^{*}\Psi(p) $ and $ \zeta\in \partial_{*}^{*}\Psi(q) $, one has
		\begin{equation*}
			\left\langle P_{q,p}\xi-\zeta; \exp_{q}^{-1}p \right\rangle_{m}\geqq0;
		\end{equation*}
		\item strictly monotone on $ \mathcal{K} $, iff for any $ p,q\in \mathcal{K} $, $\xi\in \partial_{*}^{*}\Psi(p) $ and $ \zeta\in \partial_{*}^{*}\Psi(q) $, one has
		\begin{equation*}
			\left\langle  P_{q,p}\xi-\zeta; \exp_{q}^{-1}q \right\rangle_{m}>0.
		\end{equation*}
	\end{enumerate}
\end{definition}
\noindent
In the following theorem, we discuss an important characterization of $\partial_{*}^{*}$-convex functions in the terms of monotonicity. 
\begin{theorem}\label{theorem3.4}
	Suppose $ \mathcal{K}(\neq\phi)\subseteq \mathbb{H} $ is a GC set and $ \Psi:\mathcal{K}\longrightarrow \mathbb{R}^{m} $ be a function such that $ \Psi_{i}:\mathcal{K}\longrightarrow \mathbb{R} $ are locally Lipschitz functions on $\mathcal{K}$ and  admit bounded convexificators $ \partial_{*}^{*}\Psi_{i}(p),~\mbox{$\forall$ $p\in \mathcal{K}$ and}~\forall~i\in M=\{1,2,...m\}$. Then, $ \Psi $ is $ \partial_{*}^{*}$-convex on $ \mathcal{K} $ iff $ \partial_{*}^{*}\Psi$ is monotone on $ \mathcal{K} $.
\end{theorem}        
\begin{proof}
	Suppose that $ \Psi $ is $ \partial_{*}^{*} $-convex $ \mathcal{K} $. Then, for any $ p,q\in \mathcal{K},~\xi\in\partial_{*}^{*}\Psi(p)~\mbox{and}~\zeta\in\partial_{*}^{*}\Psi(q)$, one has
	\begin{eqnarray}\label{e6}
		\Psi(p)-\Psi(q)\geqq \left\langle\zeta;\exp_{q}^{-1}p\right\rangle_{m};
	\end{eqnarray}
	and
	\begin{eqnarray}\label{e7}
		\Psi(q)-\Psi(p)\geqq \left\langle\xi;\exp_{p}^{-1}q\right\rangle_{m};
	\end{eqnarray}
	adding (\ref{e6}) and (\ref{e7}), we have
	\begin{eqnarray*}
		\left\langle P_{q,p}\xi-\zeta; \exp_{q}^{-1}p\right\rangle_{m}\geqq0.
	\end{eqnarray*}
	Hence, $ \partial_{*}^{*}\Psi $ is monotone on $ \mathcal{K} $.\\
	
	\noindent
	For converse, let $ \partial_{*}^{*}\Psi $ be a monotone on $ \mathcal{K} $ and $ z(\mu):=\exp_{q}(\mu\exp_{q}^{-1}p)~\forall~\mu\in[0,1]$.
	By the geodesic convexity of $ \mathcal{K} $, $ z(\mu)\in \mathcal{K}$,  $\forall~ \mu\in[0,1] $. By Theorem\ref{[Mean Value Theorem]}, for $ i\in M $, and $ \hat{\mu}\in(0,1) $, $\exists$ $ \tilde{\mu}_{i}\in(0,\hat{\mu}) $ and $ \bar{\mu}_{i}\in(\hat{\mu},1) $ such that for $ \tilde{\xi}_{i}\in co\partial_{*}^{*}\Psi_{i}(z(\tilde{\mu}_{i})) $ and $ \bar{\xi}_{i}\in co\partial_{*}^{*}\Psi_{i}(z(\bar{\mu}_{i})) $, one has
	\begin{equation*}
		\Psi_{i}(z(\hat{\mu}))-\Psi_{i}(z(0))=\left\langle \tilde{\xi}_{i};\exp_{z(0)}^{-1}z(\hat{\mu})\right\rangle=\hat{\mu}\left\langle\tilde{\xi}_{i};\exp_{y}^{-1}p\right\rangle;
	\end{equation*}
	and
	\begin{equation*}
		\Psi_{i}(z(1))-\Psi_{i}(z(\hat{\mu}))=\left\langle \bar{\xi}_{i}; \exp_{z(\hat{\mu})}^{-1}z(1)\right\rangle=(1-\hat{\mu})\left\langle \bar{\xi};exp_{q}^{-1}p\right\rangle.
	\end{equation*}
	By the monotonicity of $\partial_{*}^{*}\Psi$ on $ \mathcal{K} $, for any $ i\in M $ and $ \zeta_{i}\in co\partial_{*}^{*}\Psi_{i}(q) $, it follows that,
	\begin{equation*}
		\Psi_{i}(z(\hat{\mu}))-\Psi_{i}(z(0))\geq\hat{\mu}\left\langle \zeta_{i};\exp_{q}^{-1}p\right\rangle,
	\end{equation*}
	\begin{equation*}
		\Psi_{i}(z(1))-\Psi_{i}(z(\hat{\mu}))\geq(1-\hat{\mu})\left\langle\zeta_{i}:\exp_{q}^{-1}p\right\rangle.
	\end{equation*}
	On adding above inequalities, we get
	\begin{equation*}
		\Psi_{i}(p)-\Psi_{i}(q)\geq\left\langle\zeta_{i}; \exp_{q}^{-1}p\right\rangle.
	\end{equation*}
	$\implies$ $\Psi$ is $\partial_{*}^{*}$-convex on $ \mathcal{K} $.
\end{proof}

\begin{corollary}
	Suppose $ \mathcal{K}(\neq\phi)\subseteq\mathbb{H} $ is a GC set and let $ \Psi:\mathcal{K}\longrightarrow \mathbb{R}^{m} $ be a vector valued function such that $ \Psi_{i}:\mathcal{K}\longrightarrow \mathbb{R} $ are locally Lipschitz functions on $ \mathcal{K} $ and admits a bounded convexificators  $\partial_{*}^{*}\Psi(p)$   for any $ p\in \mathcal{K} $ and $ i\in M=\{1,2,...m\} $. Then, $\Psi$ is strictly  $\partial_{*}^{*}$-convex on $ \mathcal{K} $ iff  $\partial_{*}^{*}\Psi$ is strictly monotone on $ \mathcal{K} $.
\end{corollary}
\begin{proposition}\label{prop3.6}
	Suppose $ \mathcal{K}(\neq\phi)\subseteq\mathbb{H} $ is GC set and let $ \Psi:\mathcal{K}\longrightarrow\mathbb{R}^{m} $ be function such that $ \Psi_{i}:\mathcal{K}\longrightarrow\mathbb{R} $ are locally Lipschitz functions on $ \mathcal{K} $ and admits a bounded convexificator $ \partial_{*}^{*}\Psi(p) $ for any $ p\in \mathcal{K} $ and $\forall~ i\in M $. If $\Psi$ is  $\partial_{*}^{*}$-convex on $ \mathcal{K} $. Then for any $ p,q\in \mathcal{K}~\mbox{and}~\mu\in[0,1] $, one has
	\begin{equation*}
		\Psi(\exp_{q}\mu\exp_{q}^{-1}p)\leqq \Psi(q)+\mu(\Psi(p)-\Psi(q)).
	\end{equation*}
\end{proposition}
\begin{proof}
	Let $ p,q\in \mathcal{K} $ and $ z(\mu):=\exp_{q}\mu\exp_{q}^{-1}p $ for any $ \mu\in [0,1] $. By the convexity of $ \mathcal{K} $, $ z\in \mathcal{K} $. By $\partial_{*}^{*}$-convexity of $ \Psi $ on $ \mathcal{K} $, for any $\zeta\in  \partial_{*}^{*}\Psi(z)$ such that
	\begin{equation}\label{e8}
		\Psi(p)-\Psi(z)\geqq \left\langle\zeta;\exp_{z}^{-1}p\right\rangle_{m}=(1-\mu)\left\langle\zeta;\exp_{q}^{-1}p\right\rangle_{m};
	\end{equation}
	and
	\begin{equation}\label{e9}
		\Psi(q)-\Psi(z)\geqq \left\langle\zeta;\exp_{z}^{-1}q\right\rangle_{m}=-\mu\left\langle\zeta;\exp_{q}^{-1}p\right\rangle_{m};
	\end{equation}
	from (\ref{e8}) and (\ref{e9}), we have
	\begin{equation*}
		\Psi(z)\leqq \mu\Psi(p)+(1-\mu)\Psi(q)
	\end{equation*}
	that is,
	\begin{equation*}
		\Psi(\exp_{q}\mu\exp_{q}^{-1}p)\leqq \Psi(q)+\mu(\Psi(p)-\Psi(q)).
	\end{equation*}
\end{proof}
\begin{proposition}
	Suppose $ \mathcal{K}(\neq\phi)\subseteq\mathbb{H} $ is GC set and let $ \Psi:\mathcal{K}\longrightarrow\mathbb{R}^{m} $ be a function such that $ \Psi_{i}:\mathcal{K}\longrightarrow\mathbb{R} $ are locally Lipschitz functions on $ \mathcal{K} $ and for any $ p\in \mathcal{K} $ admits a bounded convexificator $ \partial_{*}^{*}\Psi(p) $ $\forall~ i\in M $. If $\Psi$ is strictly  $\partial_{*}^{*}$-convex on $ \mathcal{K} $. Then for any $ p,q\in \mathcal{K}~\mbox{and}~\mu\in[0,1] $, one has
	\begin{equation*}
		\Psi(\exp_{q}\mu\exp_{q}^{-1}p)< \Psi(q)+\mu(\Psi(p)-\Psi(q)).
	\end{equation*}
\end{proposition}
\begin{proof}
	The proof is analogous to the proposition \ref{prop3.6}.
\end{proof}
\section{Vector variational inequality problems using convexificators}
In this section, we consider the VVIP in terms of the convexificators on the Hadamard manifold and construct an example in support of definition of convexificators and existence of Stampacchia $\partial_{*}^{*}$-VVI. Furthermore, we establish the relations among Stampacchia $\partial_{*}^{*}$-VVI, the Minty type $\partial_{*}^{*}$-VVI and VOP.\\
Suppose $\mathcal{K}(\neq\phi)\subseteq \mathbb{H}$ is a set and let $\Psi:\mathcal{K}\longrightarrow \mathbb{R}^{m}$ be a vector valued function. Then we define:\\

\noindent
{\bf Stampacchia $\partial_{*}^{*}$-VVI :} Find $ \bar{p}\in \mathcal{K} $, such that for any $ q\in \mathcal{K}$, $\exists~ \xi\in\partial_{*}^{*}\Psi(\bar{p}) $, one has
\begin{equation*}
	\left\langle\xi;\exp_{\bar{p}}^{-1}q\right\rangle_{m}\notin-\mathbb{R}_{+}^{m}\setminus\{0\};
\end{equation*}
or
\begin{equation*}
	\left(\left\langle\xi_{1}; \exp_{\bar{p}}^{-1}q\right\rangle,\left\langle\xi_{2}; \exp_{\bar{p}}^{-1}q\right\rangle,...,\left\langle\xi_{m}; \exp_{\bar{p}}^{-1}q\right\rangle\right)\notin-\mathbb{R}_{+}^{m}\setminus\{0\}.
\end{equation*}
{\bf Minty $\partial_{*}^{*}$-VVI :} Find $\bar{p}\in \mathcal{K}$ such that for any $ q\in \mathcal{K} $ and $\xi\in \partial_{*}^{*}\Psi(q)$, one has
\begin{equation*}
	\left\langle\xi;\exp_{q}^{-1}\bar{p}\right\rangle_{m}\notin\mathbb{R}_{+}^{m}\setminus\{0\};
\end{equation*}
or
\begin{equation*}
	\left(\left\langle\xi_{1}; \exp_{q}^{-1}\bar{p}\right\rangle,\left\langle\xi_{2}; \exp_{q}^{-1}\bar{p}\right\rangle,...,\left\langle\xi_{m}; \exp_{q}^{-1}\bar{p}\right\rangle\right)\notin\mathbb{R}_{+}^{m}\setminus\{0\}.
\end{equation*}
In the following example, we show the existence of convexificators for Hadamard manifolds and Stampacchia $\partial_{*}^{*}$-VVI.
\begin{example}
	Let $\mathbb{H}=\{(p_{1},p_{2})\in \mathbb{R}^{2}: p_{2}>0\}$ be a Poincaré plane endowed with the Riemannian metric $g_{i,j}(p_{1},p_{2})=(\frac{1}{p_{2}^{2}})\delta_{i,j}$ for $i=1,2$, where $\delta_{i,j}$ denotes the Kronecker delta. Then, $\mathbb{H}$ is a Hadamard manifold with negative section curvature. The geodesic passing at moment $t=0$, through the point $p=(p_{1},p_{2})$, tangent to the vector $w=(w_{1},w_{2})\in T_{p}\mathbb{H}$ (see, [\citealp{udriste2013convex},pp 20]) is
	\begin{eqnarray*}
		\gamma_{w}(t)=(p_{1},p_{2}e^{t}),~for~w_{1}=0,~w_{2}=p_{2}~\mbox{ and for all}~t\in [0,\infty).
	\end{eqnarray*}
	Consider the function $\Psi:\mathbb{H}\longrightarrow\mathbb{R}^{2}$ by 
	\begin{eqnarray*}
		\Psi(p)=(\Psi_{1}(p),\Psi_{2}(p))=(|p_{1}|+p_{2}\ln p_{2}, |p_{1}|+p_{2}^{2})
	\end{eqnarray*}
	Since $\exp_{p}(tw)=\gamma_{tw}(1)=\gamma_{w}(t)=(p_{1},p_{2}e^{t})$ with the velocity vector $\gamma'_{w}(0)=(0,p_{2})\in T_{p}\mathbb{H}$.
	\begin{eqnarray*}
		\Psi_{1}(\exp_{p}tw)-\Psi_{1}(p)=\Psi_{1}(p_{1},p_{2}e^{t})-\Psi_{1}(p_{1},p_{2})
	\end{eqnarray*}
	\begin{eqnarray*}
		~~~~~~~~~~~~~~~~~~~~~~	=p_{2}\ln p_{2}(e^{t}-1)+p_{2}te^{t}
	\end{eqnarray*}
	\begin{eqnarray*}
		\frac{\Psi_{1}(\exp_{p}tw)-\Psi_{1}(p)}{t}=p_{2}\ln p_{2}(\frac{e^{t}-1}{t})+p_{2}e^{t}
	\end{eqnarray*}
	taking $\limsup$ and $\liminf$ as $t\rightarrow 0$, we have
	\begin{eqnarray*}
		\Psi_{1}^{+}(p;w)=\limsup_{t\rightarrow 0^{+}}\frac{\Psi_{1}(\exp_{p}tw)-\Psi_{1}(p)}{t}\geq p_{2}\ln p_{2}+p_{2}
	\end{eqnarray*}
	\begin{eqnarray*}
		\Psi_{1}^{-}(p;w)=\liminf_{t\rightarrow 0^{+}}\frac{\Psi_{1}(\exp_{p}tw)-\Psi_{1}(p)}{t}\leq p_{2}\ln p_{2}+p_{2}
	\end{eqnarray*}
	Hence, the convexificators of $\Psi_{1}$ at $p$ is given as follows:
	\begin{eqnarray*}
		\partial_{*}^{*}\Psi_{1}(p)=\left\{ \begin{array}{lll}
			\{(1,1+\ln p_2)\}, \quad &p_1>0,
			\\\{(1,1+\ln p_2), (-1,1+\ln p_2)\}, \quad &p_1=0,
			\\\{(-1,1+\ln p_2)\}, \quad &p_1<0.
		\end{array}
		\right.
	\end{eqnarray*}
	Now, for function $\Psi_{2}$, we have
	\begin{eqnarray*}
		\Psi_{2}(\exp_{p}tw)-\Psi_{2}(p)=\Psi_{2}(p_1, p_2e^t)-\Psi_{2}(p_1,p_2)
	\end{eqnarray*}
	\begin{eqnarray*}
		\frac{	\Psi_{2}(\exp_{p}tw)-\Psi_{2}(p)}{t}=p_{2}^{2}\frac{(e^{2t}-1)}{t}.
	\end{eqnarray*}
	Taking $\limsup$ and $\liminf$ as $t\rightarrow 0$, we have
	\begin{eqnarray*}
		\Psi_{2}^{+}(p;w)\geq 2p_{2}^{2}=p_{2}\cdot p_{2}
	\end{eqnarray*}
	\begin{eqnarray*}
		\Psi_{2}^{-}(p;w)\leq 2p_{2}^{2}=p_{2}\cdot p_{2}
	\end{eqnarray*}
	Hence, the convexificators of $\Psi_{2}$ at $p$ is given as follows:
	\begin{eqnarray*}
		\partial_{*}^{*}\Psi_{2}(p)=\left\{ \begin{array}{lll}
			\{(1,2p_2)\}, \quad &p_1>0,
			\\\{(1,2p_2), (-1,2p_2)\}, \quad &p_1=0,
			\\\{(-1,2p_2)\}, \quad &p_1<0.
		\end{array}
		\right.
	\end{eqnarray*}
	For any $q\in \mathbb{H}$ and $p=(0,1)$, $\xi_{11}:=(1,1),~\xi_{12}:=(-1,1)\in \partial_{*}^{*}\Psi_{1}(0,1)$ and $\xi_{21}:=(1,2),~\xi_{22}:=(-1,2)\in \partial_{*}^{*}\Psi_{2}(0,1)$, we have 
	\begin{eqnarray*}
		\left\langle \xi_{11};\exp_{p}^{-1}q\right\rangle=\frac{q_2}{e};~\left\langle\xi_{12};\exp_{p}\right\rangle=\frac{q_2}{e};
	\end{eqnarray*}
	\begin{eqnarray*}
		\left\langle \xi_{21}; \exp_{p}^{-1}q\right\rangle=\frac{2q_{2}}{e};~ \left\langle \xi_{22};\exp_{p}^{-1}q\right\rangle=\frac{2q_{2}}{e};
	\end{eqnarray*}
	which implies that, for any $q\in \mathbb{H}$, there exists $\xi\in \partial_{*}^{*}\Psi(p)$ such that 
	\begin{eqnarray*}
		\left\langle \xi; \exp_{p}^{-1}q\right\rangle_{2}\in \mathbb{R}_{+}^{2}.
	\end{eqnarray*}
	Therefore, $p=(0,1)$ is a solution of the Stampacchia $\partial_{*}^{*}$-VVI.
\end{example}

\noindent
The following proposition gives the relations between Stampacchia  $\partial_{*}^{*}$-VVI and Minty $\partial_{*}^{*}$-VVI.
\begin{proposition}
	Suppose $ \mathcal{K}(\neq\phi)\subseteq\mathbb{H} $ is a GC set and let $ \Psi:\mathcal{K}\longrightarrow\mathbb{R}^{m} $ be a function such that $ \Psi_{i}:\mathcal{K}\longrightarrow\mathbb{R} $ are locally Lipschitz functions on $ \mathcal{K} $ and for any $ p\in \mathcal{K} $ admits a bounded convexificator $ \partial_{*}^{*}\Psi_{i}(p) $ $\forall$ $ i\in M=\{1,2,...,m\} $. Also, suppose that $\Psi$ is $\partial_{*}^{*}$-convex on $ \mathcal{K} $. If $ \bar{p}\in \mathcal{K} $ is a solution of the Stampacchia $\partial_{*}^{*}$-VVIP, then $\bar{p}$ is also a solution of the Minty  $\partial_{*}^{*}$-VVIP.
\end{proposition}
\begin{proof}
	Let $\bar{p}$ be a solution of the Stampacchia $\partial_{*}^{*}$-VVIP. Then, for any $ q\in \mathcal{K} $, $\exists~\xi\in \partial_{*}^{*}\Psi(\bar{p}) $ such that 
	\begin{equation*}
		\left\langle\xi;\exp_{\bar{p}}^{-1}q\right\rangle_{m}\notin-\mathbb{R}_{+}^{m}\setminus\{0\}
	\end{equation*}
	Since $\Psi$ is $\partial_{*}^{*}$-convex on $ \mathcal{K} $, by Theorem \ref{theorem3.4}  $\partial_{*}^{*}\Psi$ is monotone over $ \mathcal{K} $, which implies that for any $ y\in \mathcal{K} $ and  $\zeta\in\partial_{*}^{*}\Psi(y) $, we have
	\begin{equation*}
		\left\langle\zeta;\exp_{q}^{-1}\bar{p}\right\rangle_{m}\notin\mathbb{R}_{+}^{m}\setminus\{0\}.
	\end{equation*}
	Hence, $\bar{p}$ is a solution of the Minty  $\partial_{*}^{*}$-VVIP.
\end{proof}

\noindent
\textbf{Vector optimization problem (VOP):}
Let $\mathcal{K}(\neq\phi)$ set of Hadamard manifold $\mathbb{H}$. And $\Psi:\mathbb{H}\longrightarrow \mathbb{R}^{m}$ be a vector valued function. Then, we consider vector optimization problem:
\begin{equation*}
	min\Psi(p)=(\Psi_{1}(p), \Psi_{2}(x),..., \Psi_{m}(p))
\end{equation*}
\begin{eqnarray*}
	such~that~ p\in \mathcal{K}
\end{eqnarray*}
Where $\Psi_{i}:\mathcal{K}\longrightarrow\mathbb{R}$ are real valued functions $\forall$ $i\in M=\{1,2,...,m\}$.\\

\begin{definition} A point $\bar{p}\in \mathcal{K}$ is said to be:
	\begin{enumerate}
		
		\item  an efficient solution of VOP if 
		\begin{equation*}
			\Psi(q)-\Psi(\bar{p})=(\Psi_{1}(q)-\Psi_{1}(\bar{p}),\Psi_{2}(q)-\Psi_{2}(\bar{p}),...,\Psi_{m}(q)-\Psi_{m}(\bar{p}))\notin -\mathbb{R}^{m}_{+}\setminus\{0\}~\forall~q\in \mathcal{K}.
		\end{equation*}
		\item   a weakly efficient solution of VOP if 
		\begin{equation*}
			\Psi(q)-\Psi(\bar{p})=(\Psi_{1}(q)-\Psi_{1}(\bar{p}),\Psi_{2}(q)-\Psi_{2}(\bar{p}),...,\Psi_{m}(q)-\Psi_{m}(\bar{p}))\notin -int\mathbb{R}^{m}_{+}~\forall~q\in \mathcal{K}.
		\end{equation*}
	\end{enumerate}
\end{definition}
\noindent
Remark: Efficient solution $\implies$ Weakly efficient solution.\\

\noindent
The following theorem discusses a relationship between the Stampacchia $\partial_{*}^{*}$-VVIP and the minimal point of the VOP. 
\begin{theorem}
	Suppose $ \mathcal{K}(\neq\phi)\subseteq\mathbb{H}$ is a GC set and let $\Psi:\mathcal{K}\longrightarrow\mathbb{R}^{m}$ be a function such that $\Psi_{i}:\mathcal{K}\longrightarrow\mathbb{R}$ are locally Lipschitz functions at $\bar{p}\in \mathcal{K}$ and admit bounded convexificators $\partial_{*}^{*}\Psi(\bar{p})$ $\forall~ i\in M=\{1,2,...,m\} $. Suppose that $\Psi$ is $ \partial_{*}^{*} $-convex at $\bar{p}$ over $ \mathcal{K} $. If $\bar{p}$ is a solution of the Stampacchia $\partial_{*}^{*}$-VVIP, then $\bar{p}$ is also a minimal point of the VOP.
\end{theorem}
\begin{proof}
	On contrary, suppose $\bar{p}$ is not a minimal point of the VOP. Then, $\exists~\tilde{p}$ such that 
	\begin{equation*}
		\Psi(\tilde{p})-\Psi(\bar{p})\in -\mathbb{R}_{+}^{m}\setminus\{0\}.
	\end{equation*}
	By $\partial_{*}^{*}$-convexity of $\Psi$ at $\bar{p}$ over $ \mathcal{K} $, we have
	\begin{equation*}
		\left\langle\xi;\exp_{\bar{p}}^{-1}\tilde{p}\right\rangle_{m}\in -\mathbb{R}_{+}^{m}\setminus\{0\}.
	\end{equation*}
	This contradicts to the fact that $\bar{p}$ is a solution of the Stampacchia $\partial_{*}^{*}$-VVIP.
\end{proof}

\noindent
In the following theorem , we show that the Minty $\partial_{*}^{*}$-VVIP is a necessary and sufficient condition for the VOP.
\begin{theorem}
	Suppose $ \mathcal{K}(\neq\phi)\subseteq\mathbb{H} $ is a GC set and $\Psi:\mathcal{K}\longrightarrow\mathbb{R}^{m}$ be a function such that $\Psi_{i}:\mathcal{K}\longrightarrow\mathbb{R}$ are locally Lipschitz functions on $ \mathcal{K} $ and for any $ p\in \mathcal{K} $ admits bounded convexificators $\partial_{*}^{*}\Psi(p)$, $\forall~ i\in M $. Suppose that $\Psi$ is $\partial_{*}^{*}$-convex on $ \mathcal{K} $. Then, $\bar{p}$ is a solution of the Minty $\partial_{*}^{*}$-VVIP iff $\bar{p}\in \mathcal{K}$ is a solution of the VOP.
\end{theorem}
\begin{proof}
	Contrary suppose that $\bar{p}$ is not a vector minimal point of the VOP. Then, $\exists~\tilde{p}\in \mathcal{K}$ such that 
	\begin{equation}\label{e10}
		\Psi(\tilde{p})-\Psi(\bar{p})\in -\mathbb{R}^{m}_{+}\setminus\{0\}.
	\end{equation}
	By geodesic convexity of $ \mathcal{K} $, $ p(\lambda):=\exp_{\bar{p}}\lambda\exp_{\bar{p}}^{-1}\tilde{p}\in \mathcal{K} $ for any $\lambda\in [0,1]$.\\ Since, $\Psi$ is $\partial_{*}^{*}$-convex on $\mathcal{K}$, by [Proposition\ref{prop3.6}]. We have
	\begin{equation*}
		\Psi(\exp_{\bar{p}}\lambda\exp_{\bar{p}}^{-1}\tilde{p})-\Psi(\bar{p})\leqq\lambda(\Psi(\tilde{p})-\Psi(\bar{p}));
	\end{equation*}
	or equivalently, for any $ i\in M~\mbox{and}~\lambda\in(0,1)$, one has
	\begin{equation*}
		\Psi_{i}(\exp_{\bar{p}}\lambda\exp_{\bar{p}}^{-1}\tilde{p})-\Psi_{i}(\bar{p})\leqq\lambda(\Psi_{i}(\tilde{p})-\Psi_{i}(\bar{p})).
	\end{equation*}
	By Theorem \ref{[Mean Value Theorem]}, for any $ i\in M $, $\exists~ \hat\lambda_{i}\in (0,\lambda) $ and $\hat\xi_{i}\in\mbox{co}\partial_{*}^{*}\Psi(p(\hat\lambda_{i}))$ such that
	\begin{equation*}
		\Psi_{i}(\exp_{\bar{p}}\lambda\exp_{\bar{p}}^{-1}\tilde{p})-\Psi_{i}(\bar{p})=\left\langle\hat\xi_{i};\lambda P_{p(\hat\lambda_{i}),\bar{p}}\exp_{\bar{p}}^{-1}\tilde{p}\right\rangle;
	\end{equation*}
	which implies that, for any $ i\in M $, we have
	\begin{equation}\label{e11}
		\left\langle\hat\xi_{i}; P_{p(\hat\lambda_{i}),\bar{p}}\exp_{\bar{p}}^{-1}\tilde{p}\right\rangle\leq\Psi_{i}(\tilde{p})-\Psi_{i}(\bar{p}).
	\end{equation}
	Now, there are two possible cases:\\
	\noindent
	{\bf Case(1):}  When $ \hat\lambda_{1}=\hat\lambda_{2}=...=\hat\lambda_{m}=\hat\lambda $. Multiplying both side of (\ref{e11}) by $ \hat\lambda $, for any $ i\in M $ and $ \hat\xi\in \partial_{*}^{*}\Psi(p(\hat\lambda)) $ one has
	\begin{equation*}
		\left\langle\hat\xi_{i}; P_{p(\hat\lambda_{i}),\bar{p}}\exp_{\bar{p}}^{-1}p(\hat\lambda)\right\rangle\leq\hat\lambda(\Psi_{i}(\tilde{p})-\Psi_{i}(\bar{p})).
	\end{equation*}
	From (\ref{e10}), some $ p(\hat\lambda)\in \mathcal{K} $ and $ \hat\xi\in\mbox{co}\partial_{*}^{*}\Psi(p(\hat\lambda)) $, one has
	\begin{equation*}
		\left\langle\hat\xi;\exp_{p(\hat\lambda)}^{-1}\bar{p}\right\rangle_{m}\in\mathbb{R}^{m}_{+}\setminus\{0\}.
	\end{equation*}
	Which is a contradiction to the fact that $\bar{p}$ is a solution of the Minty $\partial_{*}^{*}$-VVI.\\
	
	\noindent
	{\bf Case(2):} When  $ \hat\lambda_{1},\hat\lambda_{2},...,\hat\lambda_{m}$ are not all equal. Without loss of generality, we take $ \hat\lambda_{1}\neq\hat\lambda_{2} $. Then, from (\ref{e6}), for some $ \hat\xi_{1}\in\mbox{co}\partial_{*}^{*}\Psi_{1}(p(\hat\lambda_{1})) $ and $ \hat\xi_{2}\in\mbox{co}\partial_{*}^{*}\Psi_{2}(p(\hat\lambda_{2})) $, one has
	\begin{equation*}
		\left\langle\hat\xi_{1}; P_{p(\hat\lambda_{1}),\bar{p}}\exp_{\bar{p}}^{-1}\tilde{p}\right\rangle\leq\Psi_{1}(\tilde{p})-\Psi_{1}(\bar{p});
	\end{equation*}
	and
	\begin{equation*}
		\left\langle\hat\xi_{2}; P_{p(\hat\lambda_{2}),\bar{p}}\exp_{\bar{p}}^{-1}\tilde{p}\right\rangle\leq\Psi_{2}(\tilde{p})-\Psi_{2}(\bar{p}).
	\end{equation*}
	Since $ \Psi_{1} $ and $ \Psi_{2} $ are $\partial_{*}^{*}$-convex on $ \mathcal{K} $, by Theorem \ref{theorem3.4}, for any $ \hat\xi_{12}\in\mbox{co}\partial_{*}^{*}\Psi_{1}(p(\hat\lambda_{1})) $ and $ \hat\xi_{21}\in\mbox{co}\partial_{*}^{*}\Psi_{2}(p(\hat\lambda_{2})) $, one has
	\begin{equation*}
		\left\langle\hat\xi_{1}-\hat\xi_{12};\exp_{p(\hat\lambda_{1})}^{-1}p(\hat\lambda_{2})\right\rangle\geq0;
	\end{equation*}
	and
	\begin{equation*}
		\left\langle\hat\xi_{2}-\hat\xi_{21},\exp_{p(\hat\lambda_{2})}^{-1}p(\hat\lambda_{1})\right\rangle\geq0.
	\end{equation*}
	If $ \hat\lambda_{1}-\hat\lambda_{2}>0 $, it follows that
	\begin{equation*}
		\left\langle\hat\xi_{12}; P_{p(\hat\lambda_{1}),\bar{p}}\exp_{\bar{p}}^{-1}\tilde{p}\right\rangle\leq\Psi_{1}(\tilde{p})-\Psi_{1}(\bar{p}).
	\end{equation*}
	If  $ \hat\lambda_{2}-\hat\lambda_{1}>0 $, it follows that
	\begin{equation*}
		\left\langle\hat\xi_{21}; P_{x(\hat\lambda_{2}),\bar{p}}\exp_{\bar{p}}^{-1}\tilde{p}\right\rangle\leq\Psi_{2}(\tilde{p})-\Psi_{2}(\bar{p}).
	\end{equation*}
	Therefore, for $ \hat\lambda_{1}\neq\hat\lambda_{2} $, setting $\hat\lambda:=\{\hat\lambda_{1},\hat\lambda_{2}\}$, for any $ i=1,2 $, $\exists~ \hat\xi_{i}\in \mbox{co}\partial_{*}^{*}\Psi_{i}(p(\hat\lambda)) $ such that
	\begin{equation*}
		\left\langle\hat\xi_{i};   P_{p(\hat\lambda),\bar{p}}\exp_{\bar{p}}^{-1}\tilde{p}\right\rangle\leq\Psi_{i}(\tilde{p})-\Psi_{i}(\bar{p}).
	\end{equation*}
	On continuing the above process, we get $ \bar\lambda\in (0,\lambda) $ such that $\bar\lambda:=\min\{\hat\lambda_{1},\hat\lambda_{2},...,\hat\lambda_{m}\}$ and  $ \bar\xi_{i}\in \mbox{co}\partial_{*}^{*}\Psi_{i}(p(\bar\lambda))$ such that
	\begin{equation*}
		\left\langle\bar\xi_{i};   P_{p(\bar\lambda),\bar{p}}\exp_{\bar{p}}^{-1}\tilde{p}\right\rangle\leq\Psi_{i}(\tilde{p})-\Psi_{i}(\bar{p})~~\forall ~i\in M.
	\end{equation*}
	Multiplying the above inequality by $ \bar\lambda $, one has
	\begin{equation*}
		\left\langle\bar\xi_{i};   -\exp_{p(\bar\lambda)}^{-1}\bar{p}\right\rangle\leq\bar\lambda(\Psi_{i}(\tilde{p})-\Psi_{i}(\bar{p}))
	\end{equation*}
	By (\ref{e10}), for some $ p(\bar\lambda)\in \mathcal{K} $ and $ \bar\lambda:=(\bar\xi_{1},\bar\xi_{2},...,\bar\xi_{m})\in \partial_{*}^{*}\Psi(p(\bar\lambda)) $ one has
	\begin{equation*}
		\left\langle\bar\xi; \exp_{p(\bar\lambda)}^{-1}\bar{p}\right\rangle_{m}\in \mathbb{R}_{+}^{m}\setminus\{0\}.
	\end{equation*}
	Which contradicts to the fact that the Minty $\partial_{*}^{*}$-VVI.\\
	
	\noindent
	For converse, suppose that $\bar{x}$ is not a solution of the Minty $\partial_{*}^{*}$-VVI. Then, $\exists~\tilde{p}\in \mathcal{K}$ and $\xi\in\partial_{*}^{*}\Psi(\tilde{p})$ such that
	\begin{equation*}
		\left\langle\xi;\exp_{\tilde{p}}^{-1}\bar{p}\right\rangle_{m}\in \mathbb{R}_{+}^{m}\setminus\{0\}.
	\end{equation*}
	By $\partial_{*}^{*}$-convexity of $\Psi$ on $ \mathcal{K} $, we have 
	\begin{equation*}
		\Psi(\tilde{p})-\Psi(\bar{p})\in -\mathbb{R}_{+}^{m}\setminus\{0\};
	\end{equation*}
	a contradiction to the fact that $\bar{p}$ is solution of the VOP.
\end{proof}

\noindent
\section{Weak vector variational inequalities using convexificators}
\noindent
In this section, we first consider the weak formulations of the Stampacchia and Minty $\partial_{*}^{*}$-VVI and establish their relations with the weak minimal point of VOP.\\

\noindent
{\bf Stampacchia $\partial_{*}^{*}$-WVVI:} Find $\bar{p}\in \mathcal{K}$ such that, for any $ q\in \mathcal{K} $, $\exists~\xi\in \partial_{*}^{*}\Psi(\bar{p})$ such that
\begin{equation*}
	\left\langle\xi;\exp_{\bar{p}}^{-1}q\right\rangle_{m}\notin-int\mathbb{R}_{+}^{m}.
\end{equation*}
{\bf Minty $\partial_{*}^{*}$-VVI:} Find $\bar{p}\in \mathcal{K}$ such that for any $ q\in \mathcal{K} $ and $\xi\in\partial_{*}^{*}\Psi(q)$, one has
\begin{equation*}
	\left\langle\xi;\exp_{q}^{-1}\bar{p}\right\rangle_{m}\notin int\mathbb{R}_{+}^{m}.
\end{equation*}
The following theorem gives a necessary and sufficient condition for a point to be weak vector minimal point of the VOP in terms of the Stampacchia $\partial_{*}^{*}$-WVVI.
\begin{theorem}\label{theorem5.1}
	Suppose $ \mathcal{K}(\neq\phi)\subseteq\mathbb{H} $ is a GC set and $\Psi:\mathcal{K}\longrightarrow \mathbb{R}^{m}$ be a function such that $ \Psi_{i}:\mathcal{K}\longrightarrow \mathbb{R} $ are locally Lipschitz at point $\bar{p}\in \mathcal{K}$ and admit a bounded convexificators $\partial_{*}^{*}\Psi_{i}(\bar{p})$, $\forall~ i\in M=\{1,2,...,m\}$. Also suppose that $\Psi$ is $\partial_{*}^{*}$-convex on $ \mathcal{K} $. Then, $\bar{p}$ is a weak vector minimal point of the VOP iff $\bar{p}$ is a solution of the Stampacchia $\partial_{*}^{*}$-WVVI. 
\end{theorem}
\begin{proof}
	Suppose that $\bar{p}$ is a weak minimal point of the VOP. Then, form any $ q\in \mathcal{K} $, one has 
	\begin{equation*}
		\Psi(q)-\Psi(p)\notin-int\mathbb{R}_{+}^{m}.
	\end{equation*}
	By the geodesic convexity of $ \mathcal{K} $, for any $\lambda\in [0,1]$ and $ y\in \mathcal{K}  $, $\exp_{\bar{x}}\lambda\exp_{\bar{p}}^{-1}q\in \mathcal{K}$, which implies that 
	\begin{equation*}
		\frac{\Psi(\exp_{\bar{p}}\lambda\exp_{\bar{p}}^{-1}q)-\Psi(\bar{p})}{\lambda}\notin-int\mathbb{R}_{+}^{m};
	\end{equation*}
	taking limit inf as $\lambda\rightarrow 0^{+}$, we have
	\begin{equation*}
		\liminf_{\lambda\rightarrow 0^{+}}\frac{\Psi(\exp_{\bar{p}}\lambda\exp_{\bar{p}}^{-1}q)-\Psi(\bar{p})}{\lambda}\notin-int\mathbb{R}_{+}^{m};
	\end{equation*}
	\begin{equation*}
		\Psi^{-}(\bar{p};\exp_{\bar{p}}^{-1}q):=(\Psi^{-}_{1}(\bar{p};\exp_{\bar{p}}^{-1}q),\Psi_{2}^{-}(\bar{p};\exp_{\bar{p}}^{-1}q),...,\Psi_{m}^{-}(\bar{p};\exp_{\bar{p}}^{-1}q))\notin-int\mathbb{R}_{+}^{m},~\forall q\in \mathcal{K}.
	\end{equation*}
	Since, $ \Psi_{i} $ admit bounded convexificators $\partial_{*}^{*}\Psi_{i}(\bar{p})$ $\forall i\in M $, for any $ q\in \mathcal{K} $,  $\exists\bar{\xi}\in \partial_{*}^{*}\Psi_{i}(\bar{p})$ such that
	\begin{equation*}
		\left\langle\bar{\xi};\exp_{\bar{p}}^{-1}q\right\rangle_{m}\notin-int\mathbb{R}_{+}^{m}.
	\end{equation*} 
	Hence, $\bar{p}$ is a solution of the Stampacchia $\partial_{*}^{*}$-WVVI.\\
	
	\noindent
	For converse, contrary suppose $\bar{p}$ is not a weak vector minimal point of the VOP. Then  $\exists~\tilde{p}\in \mathcal{K}$ such that 
	\begin{equation*}
		\Psi(\tilde{p})-\Psi(\bar{p})\in -int\mathbb{R}_{+}^{m}.
	\end{equation*}
	By $\partial_{*}^{*}$-convexity of $\Psi$ at $\bar{p}$ over $ \mathcal{K} $, for any $\bar{\xi}\in \partial_{*}^{*}\Psi(\bar{p})$, one have
	\begin{equation*}
		\left\langle\bar{\xi};\exp_{\bar{p}}^{-1}p\right\rangle_{m}\in-int\mathbb{R}_{+}^{m}
	\end{equation*}
	a contradiction to the fact that $\bar{p}$ is a solution of the Stampacchia $\partial_{*}^{*}$-WVVI.
\end{proof}
\noindent
The following theorem gives the condition under which Stampacchia $\partial_{*}^{*}$-WVVI and Minty $\partial_{*}^{*}$-WVVI becomes equivalent.
\begin{theorem}
	Suppose $ \mathcal{K}(\neq\phi)\subseteq\mathbb{H} $ is a GC set and let $\Psi:\mathcal{K}\longrightarrow \mathbb{R}^{m}$ be a function such that $\Psi_{i}:K\longrightarrow\mathbb{R}$ are locally Lipschitz on $ \mathcal{K} $ and admit bounded convexificator $\partial_{*}^{*}\Psi_{i}(\bar{p})$ for any $ \bar{p}\in \mathcal{K} $,$\forall~ i\in M=\{1,2,...,m\} $. Also, suppose that $\Psi$ is $\partial_{*}^{*}$-convex on $ \mathcal{K} $. Then, $\bar{p}$ is solution of the Minty $\partial_{*}^{*}$-WVVI iff $\bar{p}$ is a solution of the Stampacchia $\partial_{*}^{*}$-WVVI. 
\end{theorem} 
\begin{proof}
	Suppose that $\bar{p}$ is a solution of the Minty $\partial_{*}^{*}$-WVVI, consider any sequence \{$\lambda_{k}$\}$\subset (0,1]$ such that $\lambda_{k}\rightarrow 0$ as $ k\rightarrow \infty$. By the geodesic convexity of $ \mathcal{K}$, for any $ q\in \mathcal{K} $, $ q_{k}=\exp_{\bar{p}}\lambda_{k}\exp_{\bar{p}}^{-1}q\in \mathcal{K} $. Since $\bar{p}$ is the solution of Minty $\partial_{*}^{*}$-WVVI, $\exists~\xi_{k}\in\partial_{*}^{*}\Psi(q_{k})$
	\begin{equation*}
		\left\langle\xi_{k};\exp_{q_{k}}^{-1}\bar{p}\right\rangle_{m}\notin int\mathbb{R}_{+}^{m}.
	\end{equation*}
	Since, $\Psi_{i}$ are locally Lipschitz and admit bounded convexificators on $ \mathcal{K} $ for all $ i\in M $, there exists $ d>0 $ such that $\|\xi_{k}\|\leq d$ which implies that the sequence $\{\xi_{k_{i}}\}\subset\partial_{*}^{*}\Psi_{i}(q_{k})$ converges to $\xi_{i}$ for all $ i\in M $. For any $ q\in \mathcal{K} $, the convexificator $\partial_{*}^{*}\Psi_{i}(q)$ are closed for all $ i\in M $. It follows that $ q_{k}\rightarrow q~\mbox{and}~\xi_{k_{i}}\rightarrow\xi_{i}~\mbox{as}~k\rightarrow\infty~\mbox{with}~\xi_{i}\in\partial_{*}^{*}\Psi_{i}(\bar{p})~\mbox{for all}~i\in M $. Therefore, for any $ y\in \mathcal{K} $, $\exists~ \xi\in\partial_{*}^{*}\Psi_{i}(\bar{p}) $ such that
	\begin{equation*}
		\left\langle\xi;\exp_{\bar{p}}^{-1}q\right\rangle_{m}\notin -int\mathbb{R}_{+}^{m}.
	\end{equation*} 
	Hence, $\bar{p}$ is a solution of the Stampacchia $\partial_{*}^{*}$-WVVI.\\
	For converse, suppose $\bar{p}$ is a solution of the Stampacchia  $\partial_{*}^{*}$-WVVI. Then, for any $ q\in \mathcal{K} $, $\exists~ \bar{\xi}\in \partial_{*}^{*}\Psi(\bar{p}) $ such that 
	\begin{equation*}
		\left\langle\bar\xi;\exp_{\bar{p}}^{-1}q\right\rangle_{m}\notin -int\mathbb{R}_{+}^{m}.
	\end{equation*}
	Since, $\Psi$ is $\partial_{*}^{*}$-convex on $ \mathcal{K} $, by Theorem \ref{theorem3.4}, we get $\partial_{*}^{*}\Psi$ is monotone on $ \mathcal{K} $, which implies
	\begin{equation*}
		\left\langle\xi;\exp_{q}^{-1}\bar{p}\right\rangle_{m}\notin int\mathbb{R}_{+}^{m}
	\end{equation*}
	for any $ q\in \mathcal{K}~\mbox{and}~\xi\in\partial_{*}^{*}\Psi(q) $. Hence, $\bar{p}$ is a solution of the Minty $\partial_{*}^{*}$-WVVI.
\end{proof}
\noindent
The following theorem gives the condition for a weak vector minimal point to be a vector minimal point of the VOP.
\begin{theorem}
	Suppose $ \mathcal{K}(\neq\phi)\subseteq\mathbb{H} $ is a GC set and $\Psi:\mathcal{K}\longrightarrow \mathbb{R}^{m}$ be a function such that $ \Psi_{i}:\mathcal{K}\longrightarrow \mathbb{R} $ are local Lipschitz at $\bar{p}\in \mathcal{K}$ and admit bounded convexificators $\partial_{*}^{*}\Psi_{i}(\bar{p})$, $\forall~ i\in M=\{1,2,...,m\} $. Also suppose that $ \Psi $ is strictly $\partial_{*}^{*}$-convex at $\bar{p}$ over $ \mathcal{K} $. Then, $\bar{p}$ is a vector minimal point of VOP iff $\bar{p}$ is a weak vector minimal point of the VOP.
\end{theorem}
\begin{proof}
	Obviously, every vector minimal point is also a weak minimal point of the VOP.\\
	Conversely, suppose that $\bar{p}$ is weak minimal point of the VOP but not a vector minimal point of the VOP. Then, $\exists~\tilde{p}\in \mathcal{K}$ such that
	\begin{equation*}
		\Psi(\tilde{p})-\Psi(\bar{p})\in -int\mathbb{R}_{+}^{m}.
	\end{equation*}
	By strict $\partial_{*}^{*}$-convexity of $\Psi$ at $\bar{p}$ over $ \mathcal{K} $, for any $\bar{\xi}\in \partial_{*}^{*}\Psi(\bar{p})$, one have
	\begin{equation*}
		\left\langle \bar{\xi}; \exp_{\bar{p}}^{-1}\tilde{p}\right\rangle_{m}\in -int\mathbb{R}_{+}^{m}
	\end{equation*}
	which implies that $\bar{p}$ is not a solution of Stampacchia $\partial_{*}^{*}$-WVVI. By Theorem\ref{theorem5.1}, $\bar{p}$ is not a weak vector minimal point of the VOP. This contradiction leads to the results. 
\end{proof}

\section{Conclusion}
In this paper, we have formulated the concept of convexificators for Hadamard manifolds which is a weaker version of the notion of sub-differentials and proved the mean value theorem for them. Furthermore, we discussed the characterizations of the $\partial_{*}^{*}$-convex functions in terms of monotonicity. We defined Stampacchia and Minty type $\partial_{*}^{*}$-VVI using convexificators and construct an example in support of their existence and also established the relationships between their solutions and vector minimal point of VOP.\\
\noindent
Due to the use of convexificators, the  results of this paper are precise and more general than the similar results already presented in the literature. Nevertheless, the existence results of $\partial_{*}^{*}$-VVI is still an open problem. They also may be extended on Riemannian manifolds using some more assumptions. Moreover, some related problems like fixed point problems, complementarity problem and equilibrium problems using the concept of convexificators can be explored in the future.

\bibliography{bibliography}
\bibliographystyle{apalike}

\end{document}